\documentclass[]{article}

\usepackage{amsmath,amsthm,mathrsfs,amssymb,enumerate}
\usepackage{url}

\newtheorem{thm}{Theorem}[section]
\newtheorem{rem}[thm]{Remark}
\newtheorem{conj}[thm]{Conjecture}

\numberwithin{equation}{section}

\title{\large \bf Concavity of some entropies}

\author{\bf Ioan RA\c{S}A\\
\small Department of Mathematics, Technical University of Cluj-Napoca, Romania\\ \small Email address: Ioan.Rasa@math.utcluj.ro}
\date{}

\begin{document}

\maketitle
%\section*{\small SESSION: Mathematical Analysis and Applications}
\begin{abstract}
It is well-known that the Shannon entropies of some parameterized probability distributions are concave functions with respect to the parameter. In this paper we consider a family of such distributions (including the binomial, Poisson, and negative binomial distributions) and investigate the concavity of the Shannon, R\'{e}nyi, and Tsallis entropies of them.
\end{abstract}

\bigskip \bigskip

\section{Introduction}

Let $c\in \mathbb{R}$, $I_c := \left [ 0, -\frac{1}{c}\right ]$ if $c<0$, and $I_c:= [0,+\infty )$ if $c \geq 0$.

For $\alpha \in \mathbb{R}$ and $k \in \mathbb{N}_0$ the binomial coefficients are defined as usual by
\begin{equation*}
{\alpha \choose k}:=\frac{\alpha (\alpha -1)\dots (\alpha-k+1)}{k!}\quad \text{if } k \in \mathbb{N}, \text{ and } {\alpha \choose 0}:=1.
\end{equation*}

Let $n> 0$ be a real number such that $n>c$ if $c\geq 0$, or $n=-cl$ with some $l\in \mathbb{N}$ if $c<0$.

For $k\in \mathbb{N}_0$ and $x\in I_c$ define
\begin{equation*}
p_{n,k}^{[c]}(x):=(-1)^k {-\frac{n}{c} \choose k}(cx)^k (1+cx)^{-\frac{n}{c}-k}, \quad \text{if } c\neq 0,
\end{equation*}
\begin{equation*}
p_{n,k}^{[0]}(x):=\lim _{c\to 0} p_{n,k}^{[c]}(x)= \frac{(nx)^k}{k!}e^{-nx}.
\end{equation*}

Details and historical notes concerning these functions can be found in~\cite{3}, \cite{7}, \cite{21} and the references therein. In particular,
\begin{equation}
\frac{d}{dx}p_{n,k}^{[c]}(x) = n \left ( p_{n+c,k-1}^{[c]}(x) - p_{n+c,k}^{[c]}(x)\right ).\label{eq:1.1}
\end{equation}

Moreover, $\sum _{k=0}^\infty p_{n,k}^{[c]}(x) = 1$, so that $\left (p_{n,k}^{[c]}(x)\right )_{k\geq 0}$ is a parameterized probability distribution. Its associated Shannon entropy is
\begin{equation*}
H_{n,c}(x):=-\sum_{k=0}^\infty p_{n,k}^{[c]}(x) \log p_{n,k}^{[c]}(x),
\end{equation*}
while the R\'{e}nyi entropy of order $2$ and the Tsallis entropy of order $2$ are given, respectively, by (see~\cite{18}, \cite{20})
\begin{equation*}
R_{n,c}(x):= -\log S_{n,c}(x); \quad T_{n,c}(x):=1-S_{n,c}(x),
\end{equation*}
where
\begin{equation*}
S_{n,c}(x) := \sum _{k=0}^\infty \left (p_{n,k}^{[c]}(x)\right )^2, \quad x\in I_c.
\end{equation*}

The cases $c=-1$, $c=0$, $c=1$ correspond, respectively, to the binomial, Poisson, and negative binomial distributions.

\section{Shannon entropy}

$H_{n,-1}$ is a concave function; this is a special case of the results of~\cite{19}; see also~\cite{6}, \cite{8}, \cite{9} and the references therein. $H_{n,0}$ is also concave; moreover, $H'_{n,0}$ is completely monotonic (see, e.g., \cite[p. 2305]{2}).

For the sake of completeness we present here the proof for the concavity of $H_{n,c}$, $c\in \mathbb{R}$. Let us consider separately the cases $c\geq 0$ and $c<0$.

\begin{thm}\label{th:2.1}
For $c\geq 0$, $H_{n,c}$ is concave and increasing on $[0,+\infty )$.
\end{thm}
\begin{proof}
Using~\eqref{eq:1.1}, it is a matter of calculus to prove that
\begin{equation*}
H'_{n,c}(x) = n \sum _{k=0} ^\infty p_{n+c,k}^{[c]}(x)\log \frac{p_{n,k}^{[c]}(x)}{p_{n,k+1}^{[c]}(x)},
\end{equation*}
which leads to
\begin{equation}
H'_{n,c}(x) = n \left (\log \frac{1+cx}{x} + \sum _{k=0}^\infty p_{n+c,k}^{[c]}(x) \log \frac{k+1}{n+ck} \right ),\label{eq:2.1}
\end{equation}
and therefore to
\begin{equation*}
H_{n,c}''(x) = -\frac{n}{x(1+cx)} + n (n+c)\sum_{k=0}^\infty p_{n+2c,k}^{[c]}(x)\log \frac{(k+2)(n+ck)}{(k+1)(n+ck+c)}.
\end{equation*}

It follows that
\begin{equation}
H''_{n,c}(x)>-\frac{n}{x(1+cx)},\quad x>0. \label{eq:2.2}
\end{equation}

Since $\log t<t-1$, $t>1$, we have also
\begin{eqnarray*}
H''_{n,c}(x) &<& -\frac{n}{x(1+cx)} + n (n+c)\sum_{k=0}^\infty p_{n+2c,k}^{[c]}(x) \frac{n-c}{(k+1)(n+ck+c)} \\ &=& -\frac{n}{x(1+cx)} \left ( \frac{c}{n} + \left ( 1-\frac{c}{n} \right ) (1+cx)^{-\frac{n}{c}} \right ) <0,
\end{eqnarray*}
so that $H_{n,c}$ is concave on $[0,+\infty )$; being positive, it is also increasing on $[0,+\infty)$.
\end{proof}

\begin{rem}
From~\eqref{eq:2.2} it follows that the functions $H_{n,0}(x)+n x \log x$ and $H_{n,c}(x) + \frac{n}{c}\left ( c x \log x - (1+cx)\log (1+cx) \right )$, $c>0$, are convex on $[0,+\infty)$.
\end{rem}

\begin{rem}
The following inequalities are valid for $x>0$ and $c\geq 0$:
\begin{equation}
\log \frac{x}{1+cx} \leq \sum _{k=0}^\infty p_{n+c,k}^{[c]}(x)\log \frac{k+1}{ck+n} \leq \log \frac{x+\frac{1}{n}}{1+cx}. \label{eq:2.3}
\end{equation}
\end{rem}

The first one follows from $H_{n,0}'(x)>0$, taking into account~\eqref{eq:2.1}, and the second is a consequence of Jensen's inequality applied to the concave function $\log x$. In particular, for $c=0$ and $n=1$ we get:
\begin{equation*}
\log x \leq \sum _{k=0}^\infty e^{-x}\frac{x^k}{k!}\log (k+1) \leq \log (x+1), \quad x>0.
\end{equation*}

The case $c<0$ can be studied with the same method as in Theorem~\ref{th:2.1}, but we present here a different approach, based on an integral representation from~\cite{10}.

\begin{thm}\label{th:2.4}
For $c<0$, $H_{n,c}$ is concave on $\left [ 0, -\frac{1}{c}\right ]$, increasing on $\left [ 0, -\frac{1}{2c}\right ]$, and decreasing on $\left [ -\frac{1}{2c}, -\frac{1}{c}\right ]$.
\end{thm}

\begin{proof}
For $c<0$ we have $n=-cl$ with $l\in \mathbb{N}$. Using~\cite[(2.5)]{10} we get
\begin{eqnarray*}
&&H_{n,c}(x) = H_{l,-1}(-cx) = -l\left [ (-cx)\log (-cx) + (1+cx)\log (1+cx) \right ] +\\&& \int _0^1 \frac{s-1}{\log s} \frac{(1+cx-cxs)^l + \left ( (1+cx)s-cx\right )^l-1-s^l}{(s-1)^2}ds.
\end{eqnarray*}

It is  matter of calculus to prove that
\begin{equation}\label{eq:2.4}
H_{n,c}''(x) = -\frac{n}{x(1+cx)}+c^2l(l-1)\int _0^1 \frac{s-1}{\log s}\left [ (1+cx-cxs)^{l-2} +\left ( \left ( 1+cx\right)s-cx \right) ^{l-2}\right ]ds.
\end{equation}

For $0<s<1$ we have $0<\frac{s-1}{\log s}<1$; moreover
\begin{equation*}
\int _0^1 \left [ (1+cx-cxs)^{l-2} + \left (\left ( 1+cx\right)s-cx  \right ) ^{l-2}  \right ]ds = \frac{1}{l-1}\left [ \frac{1-(-cx)^{l-1}}{1+cx} - \frac{1-(1+cx)^{l-1}}{cx}\right ].
\end{equation*}

Summing up, we get
\begin{equation*}
H_{n,c}''(x)<-n\frac{(1+cx)^{-\frac{n}{c}}+(-cx)^{-\frac{n}{c}}}{x(1+cx)}<0,\quad 0<x<-\frac{1}{c}.
\end{equation*}

Consequently, $H_{n,c}$ is concave on $\left [ 0,-\frac{1}{c} \right ]$. Since
\begin{equation*}
H_{n,c}\left ( -\frac{1}{2c} - t\right )  = H_{n,c} \left ( -\frac{1}{2c}+t\right ), \quad t \in \left [ 0,-\frac{1}{2c} \right ],
\end{equation*}
we conclude that $H_{n,c}$ is increasing on $\left [ 0,-\frac{1}{2c} \right ]$ and decreasing on $\left [ -\frac{1}{2c}, -\frac{1}{c} \right ]$.
\end{proof}

\begin{rem}
Let $c<0$ and $l\geq 2$. From~\eqref{eq:2.4} it follows that $H_{n,c}''(x)>-\frac{n}{x(1+cx)}$, and so the function
\begin{equation*}
H_{n,c}(x) + \frac{n}{c}\left [ cx \log x - (1+cx)\log (1+cx) \right ]
\end{equation*}
is convex on $\left [0, -\frac{1}{c} \right ]$.
\end{rem}

\begin{rem}
For $c=-1$, the method used to prove~\eqref{eq:2.3} leads to
\begin{eqnarray*}
\log \frac{x}{1-x} &<& \sum _{k=0}^n {n \choose k} x^k (1-x)^{n-k}\log \frac{k+1}{n+1-k}\\ &<& \log \left (\frac{x}{1-x} + \frac{1-(n+2)x^{n+1}}{(n+1)(1-x)} \right ), \quad 0<x<\frac{1}{2}.
\end{eqnarray*}
\end{rem}

\section{$S_{n,c}$ and Heun functions}

The following conjecture was formulated in~\cite{13}:
\begin{conj}\label{conj:3.1}
$S_{n,-1}$ is convex on $[0,1]$.
\end{conj}

Th. Neuschel~\cite{11} proved that $S_{n,-1}$ is decreasing on $\left [ 0, \frac{1}{2}\right ]$ and increasing on $\left [ \frac{1}{2}, 1\right ]$. The conjecture and the result of Neuschel can be found also in~\cite{5}.

A proof of the conjecture was given by G. Nikolov~\cite{12}, who related it with some new inequalities involving Legendre polynomials. Another proof can be found in~\cite{4}.

Using the important results of Elena Berdysheva~\cite{3}, the following extension was obtained in~\cite{17}:
\begin{thm}\label{th:3.2}
(\cite[Theorem 9]{17}). For $c<0$, $S_{n,c}$ is convex on $\left [ 0, -\frac{1}{c}\right ]$.
\end{thm}

A stronger conjecture was formulated in~\cite{14} and~\cite{17}:
\begin{conj}\label{conj:3.3}
For $c\in \mathbb{R}$, $S_{n,c}$ is logarithmically convex, i.e., $\log S_{n,c}$ is convex.
\end{conj}

It was validated for $c\geq 0$ by U. Abel, W. Gawronski and Th. Neuschel~\cite{1}, who proved a stronger result:
\begin{thm}\label{th:3.4}
(\cite{1}). For $c\geq 0$, the function $S_{n,c}$ is completely monotonic, i.e.,
\begin{equation*}
(-1)^m \left ( \frac{d}{dx} \right )^m S_{n,c}(x)>0, \quad x\geq 0, m\geq 0.
\end{equation*}

Consequently, for $c\geq 0$, $S_{n,c}$ is logarithmically convex, and hence convex.
\end{thm}

On the other hand, according to~\cite[Th. 4]{17}, $S_{n,c}$ is a solution to the differential equation
\begin{eqnarray}
x(1+cx)(1+2cx)y''(x)+\left ( 4(n+c)x(1+cx)+1 \right )y'(x) +\nonumber \\ +2n(1+2cx)y(x)=0. \label{eq:3.1}
\end{eqnarray}

Consequently, for $c\neq 0$ the function $S_{n,c}\left ( -\frac{x}{c}\right )$ is a solution to the Heun equation
\begin{equation*}
y''(x)+\left ( \frac{1}{x}+\frac{1}{x-1}+\frac{\frac{2n}{c}}{x-\frac{1}{2}} \right )y'(x)+\frac{\frac{2n}{c}x-\frac{n}{c}}{x(x-1)\left ( x-\frac{1}{2}\right )} y(x)= 0,
\end{equation*}
and $S_{n,0}$ is a solution of the confluent Heun equation:
\begin{equation*}
u''(x)+\left ( 4n+\frac{1}{x} \right )u'(x)+\frac{2nx-2n}{x(x-1)}u(x)=0.
\end{equation*}

For details, see~\cite{14}-\cite{17}.

\section{R\'{e}nyi entropy and Tsallis entropy}

\begin{thm}\label{th:4.1}
\begin{enumerate}[(i)]
\item{}For $c\geq 0$, $R_{n,c}$ is concave and increasing on $[0,+\infty )$.
\item{}$R_{n,c}'$, with $c\in \mathbb{R}$, is a solution to the Riccati equation
\begin{eqnarray*}
x(1+cx)(1+2cx)u'(x) = x(1+cx)(1+2cx)u^2(x) -\\- \left ( 4(n+c)x(1+cx)+1\right )u(x) +2n(1+2cx).
\end{eqnarray*}
\end{enumerate}
\end{thm}

\begin{proof}
\begin{enumerate}[i)]
\item{}is a direct consequence of Theorem~\ref{th:3.4}.
\item{}We have $S_{n,c} = \exp (-R_{n,c})$ and~\eqref{eq:3.1} yields
\begin{eqnarray*}
x(1+cx)(1+2cx)\left ( (R_{n,c}')^2 - R_{n,c}'' \right )-\\-\left ( 4(n+c)x(1+cx)+1 \right ) R_{n,c}'+2n(1+2cx)=0.
\end{eqnarray*}

Setting $u=R_{n,c}'$, we conclude the proof.
\end{enumerate}
\end{proof}

\begin{rem}
As far as we know, Conjecture~\ref{conj:3.3} is still open for $c<0$, so that the concavity of $R_{n,c}$, $c<0$, remains to be investigated.
\end{rem}

\begin{thm}\label{th:4.2}
\begin{enumerate}[(i)]
\item{}$T_{n,c}$ is concave. For $c\geq 0$ it is increasing on $[0,+\infty )$. For $c<0$ it is increasing on $\left [0, -\frac{1}{2c} \right ]$ and decreasing on $\left [-\frac{1}{2c}, -\frac{1}{c} \right ]$.
\item{}$T_{n,c}$ is a solution to the equation
\begin{eqnarray*}
x(1+cx)(1+2cx)u''(x)+\left ( 4(n+c)x(1+cx) + 1 \right )u'(x)+\\ + 2n(1+2cx)u(x) = 2n(1+2cx).
\end{eqnarray*}
\end{enumerate}
\end{thm}
\begin{proof}
\begin{enumerate}[(i)]
\item{}is a consequence of Theorems~\ref{th:3.2} and~\ref{th:3.4}, while
\item{}follows from~\eqref{eq:3.1}.
\end{enumerate}
\end{proof}

\section{Some inequalities}

\begin{enumerate}[a)]

\item{} The explicit expression of $S_{n,-1}$, $n\in \mathbb{N}$, is
\begin{equation*}
S_{n,-1}(x) = \sum_{k=0}^n \left ( {n \choose k}x^k(1-x)^{n-k} \right )^2, \quad x\in [0,1].
\end{equation*}

Consider also the function
\begin{equation*}
f_n(t):=\sum_{j=0}^{n-1} {n-1 \choose j}\left ( {n\choose j+1}t^{2j+1} - {n \choose j}t^{2j} \right ), \quad t\geq 1.
\end{equation*}

Since $S_{n,-1}(1-x) = S_{n,-1}(x)$, it follows that
\begin{equation*}
S_{n,-1}^{(2j+1)}\left ( \frac{1}{2} \right )=0, \quad j=0,1, \dots, n-1 .
\end{equation*}

In relation with Conjecture~\ref{conj:3.1}, it was also conjectured in~\cite{13} that
\begin{equation}
S_{n,-1}^{(2j)}\left ( \frac{1}{2} \right ) >0, \quad j=0,1, \dots, n,\label{eq:5.1}
\end{equation}

\begin{equation}
f_n^{(i)}(t) \geq 0, \quad t\geq 1; \quad i=0,1, \dots, 2n-1, \label{eq:5.2}
\end{equation}

\begin{equation}
\sum _{j=[i/2]}^{n-1} {n-1\choose j}\left ( {n \choose j+1} {2j+1 \choose i} - {n\choose j} {2j \choose i}\right )\geq 0, \quad i = 0, 1, \dots, 2n-1. \label{eq:5.3}
\end{equation}

We shall prove here these inequalities\footnote{For~\eqref{eq:5.2} with $i=0$, see also \\ \url{http://www.artofproblemsolving.com/community/c6h494060p2772738}.}.

It can be proved directly that
\begin{equation*}
S_{n,-1} \left ( \frac{1}{2}\right ) = \frac{1}{4^n}{2n \choose n}; \quad S_{n,-1}^{(2)}\left ( \frac{1}{2}\right ) = \frac{1}{4^{n-2}}{2n-2 \choose n-1};
\end{equation*}
\begin{equation*}
S_{n,-1}^{(4)}\left ( \frac{1}{2}\right ) = \frac{9}{4^{n-4}}{2n-4 \choose n-2}.
\end{equation*}

The following formula was obtained in~\cite{4}:
\begin{equation}
S_{n,-1}(x) = \sum_{k=0}^n 4^{k-n} {2n \choose n}{n \choose k}^2 {2n \choose 2k}^{-1}\left ( x-\frac{1}{2} \right )^{2k}.\label{eq:5.4}
\end{equation}

Using it we get
\begin{equation*}
S_{n,-1}^{(2j)}\left ( \frac{1}{2}\right ) = (2j)! 4^{j-n}{2n \choose n}{n \choose j}^2 {2n \choose 2j}^{-1}, \quad j=0,1,\dots ,n,
\end{equation*}
and so~\eqref{eq:5.1} is proved.

On the other hand, it is not difficult to prove that
\begin{equation}
f_n(t) = \frac{1}{2n}(t+1)^{2n-1}S_{n,-1}'\left ( \frac{t}{t+1}\right ), \quad t\geq 1. \label{eq:5.5}
\end{equation}

From~\eqref{eq:5.4} we obtain
\begin{equation}
S_{n,-1}'(x) = \sum _{k=1}^{n}2ka_{nk} \left ( x - \frac{1}{2}\right )^{2k-1} \label{eq:5.6}
\end{equation}
for certain $a_{nk}>0$. Now~\eqref{eq:5.5} and~\eqref{eq:5.6} imply
\begin{equation*}
f_n(t) = \sum _{k=1}^n c_{nk}(t-1)^{2k-1} (t+1)^{2n-2k}, \quad t\geq 1,
\end{equation*}
with suitable $c_{nk}>0$, and using Leibniz' formula we get~\eqref{eq:5.2}.

Finally, starting from the definition of $f_n(t)$, it is not difficult to infer that
\begin{equation*}
f_n^{(i)}(1)=i! \sum_{j=[i/2]}^{n-1} {n-1 \choose j} \left ( {n \choose j+1} {2j+1 \choose i} - {n \choose j}{2j \choose i} \right ).
\end{equation*}

Combined with~\eqref{eq:5.2}, this proves~\eqref{eq:5.3}.

\item{} Let $B_m : \mathscr{C}[0,1]\to \mathscr{C}[0,1]$, $B_m f(x) = \sum _{k=0}^m {m \choose k} x^k (1-x)^{m-k}$, be the classical Bernstein operators. It is well-known that if $f \in \mathscr{C}[0,1]$ is convex, then $B_m f$ is also convex.

Now let $w_n \in \mathscr{C}[0,1]$ be piece-wise linear, such that $w_n \left ( \frac{2k-1}{2n}\right ) = 0$, $k=1,\dots, n$; $w_n \left ( \frac{2k}{2n} \right ) = {n \choose k}^2 / {2n \choose 2k}$, $k = 0, 1, \dots , n$. Then $B_{2n}w_n = S_{n,-1}$, hence $B_{2n}w_n$ is convex although the graph of $w_n$ is "like a saw".

\end{enumerate}

\end{document}